\documentclass[a4paper,11pt]{amsart}

\usepackage{amsmath,amssymb}
\usepackage{amscd}
\usepackage{amsthm}
\usepackage{mathrsfs}
\usepackage{enumitem}
\usepackage{mathtools}

\theoremstyle{plain}
\newtheorem{theorem}[subsection]{Theorem}
\newtheorem{proposition}[subsection]{Proposition}
\newtheorem{lemma}[subsection]{Lemma}

\newtheorem*{theorem*}{Theorem}

\theoremstyle{definition}
\newtheorem{definition}[subsection]{Definition}
\newtheorem{conjecture}[subsection]{Conjecture}

\theoremstyle{remark}
\newtheorem{remark}[subsection]{Remark}
\newtheorem{notation}[subsection]{Notation}
\newtheorem{question}[subsection]{Question}

\newtheorem{convention}[subsection]{Convention}
\newtheorem{claim}[subsubsection]{Claim}

\numberwithin{equation}{subsection}

\DeclareMathOperator{\Aut}{Aut}
\DeclareMathOperator{\Hom}{Hom}
\DeclareMathOperator{\RC}{RatCurves}
\DeclareMathOperator{\Chow}{Chow}
\DeclareMathOperator{\im}{im}
\DeclareMathOperator{\Pic}{Pic}
\DeclareMathOperator{\NE}{NE}

\DeclareMathOperator{\id}{id}
\DeclareMathOperator{\pr}{pr}
\DeclareMathOperator{\ev}{ev}

\newcommand{\cNE}{{\overline{\NE}}}
\newcommand{\red}{_{\mathrm{red}}}
\newcommand{\zred}{_{z,\mathrm{red}}}
\newcommand{\nequiv}{\equiv _\mathrm{num}}

\newcommand{\sF}{\mathscr{F}}
\newcommand{\sG}{\mathscr{G}}
\newcommand{\sL}{\mathscr{L}}

\newcommand{\cO}{\mathcal{O}}

\newcommand{\bC}{\mathbb{C}}
\newcommand{\bF}{\mathbb{F}}

\newcommand{\bP}{\mathbb{P}}
\newcommand{\bQ}{\mathbb{Q}}
\newcommand{\bR}{\mathbb{R}}

\newcommand{\step}[1]{\par\medskip\par\noindent\textit{#1}}

\newcommand{\case}[1]{\par\medskip\par\noindent\textit{#1}}

\setenumerate{font=\normalfont}

\AtBeginDocument{%
   \def\MR#1{}
}

\title{Fano 5-folds with nef tangent bundles}
\author[A. KANEMITSU]{Akihiro KANEMITSU}
\address{Graduate School of Mathematical Sciences\\The University of Tokyo\\3-8-1 Komaba\\Meguro-ku, Tokyo 153-8914, Japan}
\email{kanemitu@ms.u-tokyo.ac.jp}

\begin{document}


\begin{abstract} 
We prove that Fano 5-folds with nef tangent bundles are rational 
homogeneous manifolds.
\end{abstract}

\maketitle

\section*{Introduction}\label{intro}

In 1979, S. Mori proved that projective manifolds with ample tangent bundles are projective spaces \cite{Mor}.
Since then, several authors have considered how to generalize Mori's result to the case where manifolds in question satisfy weaker positivity conditions.
For example, N. Mok proved that all compact K\"ahler manifolds of semipositive holomorphic bisectional curvature 
are \'etale quotients of symmetric spaces \cite{Mok1}.

A natural algebraic counterpart of Mok's result is the following:
\begin{conjecture}[Campana-Peternell Conjecture \cite{CP1}]\label{CP}
A Fano manifold $X$ with nef tangent bundle is a rational homogeneous manifold.
\end{conjecture}
Indeed, a compact K\"ahler manifold with nef tangent bundle is decomposed into 
a ``Fano part'' and a ``complex torus part'' after taking an \'etale cover by
the following theorem of 
J.-P. Demailly, T. Peternell and M. Schneider \cite{DPS}:
\begin{theorem*}[{\cite[Theorem 3.14]{DPS}}]
Any compact K\"ahler manifold $X$ with nef tangent bundle admits an \'etale cover $\widetilde{X} \to X$ with following properties:
\begin{enumerate}
\item The Albanese map $\alpha \colon \widetilde{X} \to A(\widetilde{X})$ is a smooth fibration.
\item The fibers of $\alpha$ are Fano manifolds with nef tangent bundles.
\end{enumerate}

\end{theorem*}

The objective of this paper is to check Conjecture~\ref{CP} in dimension five: 
\begin{theorem}\label{CP5}
Fano $5$-folds with nef tangent bundles are rational homogeneous manifolds.
\end{theorem}

Conjecture \ref{CP} is known to be true when $\dim X$ is at most 
four \cite{CP1,CP2,Mok,Hw} and when $X$ is a 5-fold with Picard number two or more \cite{W2}.
For further results about the Campana-Peternell conjecture, we refer the reader to the survey article \cite{MOSWW}.

For brevity, we call a Fano manifold with nef tangent bundle \emph{a CP manifold}.
In this paper, we study mainly CP manifolds with Picard number one.
Given such a manifold $X$ of dimension $n$, 
the \emph{pseudoindex} $i_X$ is defined as the minimum anticanonical degree of rational curves on $X$ (see Definition~\ref{index} below).
Assume $n\geq2$.
Then it is known that $3 \leq i_X \leq n+1$ (see e.g.\ \cite[Theorem~5.1]{W2}).
Furthermore, by the results of Cho-Miyaoka-Shepherd-Barron and Miyaoka,
the pseudoindex $i_X$ is $n+1$ (resp.\ $n$) 
if and only if $X$ is $\mathbb{P}^n$ (resp.\ $\mathbb{Q}^n$) \cite{CMSB,Mi}.
On the other hand, the smallest pseudoindex case is treated by J.-M. Hwang and N. Mok;
$X$ with pseudoindex three is isomorphic to either $\mathbb{P}^2$, $\mathbb{Q}^3$ or $K(G_2)$, where $K(G_2)$ is the $5$-dimensional contact homogeneous manifold of type $G_2$ \cite{Mok,Hw} (see also \cite{MOS,W1}).
Hence, if $n=5$, $X$ is homogeneous or the pseudoindex $i_{X}$ is four.
In this paper, we study the remaining case $i_X=4$.
Note that there is no rational homogeneous $5$-fold with Picard number one and pseudoindex four \cite[Remark~5.3]{W2}.
Hence, we shall show that there is no CP $5$-fold with Picard number one and pseudoindex four, which will complete the proof of Theorem~\ref{CP5}.

\medskip
We sketch the outline of this paper.

In Section~\ref{pre}, we review basic results  
concerning families of minimal rational curves.
Given a CP manifold $X$ with Picard number one,
an irreducible component $V$ of the scheme parametrizing  rational curves on $X$ is called a \emph{minimal rational component} if it parametrizes rational curves of minimum anticanonical degree. 
Given a minimal rational component $V$,
there exists the following diagram consisting of the two natural projections
\[
\begin{CD}
     U        @> e >>   X     \\
@V \pi VV                        \\
     V,
\end{CD}
\]
where $\pi \colon  U \to V$ is the universal family of rational curves and $e\colon  U \to X$ is the evaluation morphism. 
Since the tangent bundle is nef, the evaluation morphism $e$ is a smooth morphism \cite[Corollary~1.3]{KMM1} (see also \cite[II. Theorem~2.15 and Corollary~3.5.3]{K}).

In Sections~\ref{twoP} and \ref{fammrc}, we study the structure of $(U,X,V; e,\pi)$ above 
when $X$ has pseudoindex four.   
In this case, the evaluation morphism $e$ is of relative dimension two. 
By virtue of a result of K.~Oguiso and E.~Viehweg \cite{OV}, We can apply a similar argument as in the proof of \cite[Lemma~1.2.2]{Mok} and find an extremal ray of $\cNE (U/X)$ (see Lemma~\ref{ray}).
Then, by carefully looking at the associated extremal contraction, we prove the following theorem (see Theorem~\ref{fib} for a precise statement).

\begin{theorem}\label{fiber}
Let $X$ be a CP manifold with Picard number one and pseudoindex four.
Then the evaluation morphism of minimal rational curves is one of the following:
\begin{enumerate}
\item\label{fiber1} a smooth $\bP^2$-fibration.
\item a composite of two smooth $\bP^1$-fibrations.
\end{enumerate}
\end{theorem}
\noindent
This theorem asserts that the evaluation morphism associated with a CP manifold $X$ looks like 
the one associated with a rational homogeneous manifold.

Finally, in Section~\ref{pf_CP5}, we complete the proof of Theorem~\ref{CP5}.
Recently, the following strategy to study the Campana-Peternell conjecture was proposed in \cite{MOSW} (see also \cite[Section~6]{MOSWW}):
\begin{enumerate}
\item Prove that the CP manifolds with ``maximal'' Picard number are complete flag manifolds;
\item Prove that any CP manifold is dominated by a CP manifold with ``maximal'' Picard number.
\end{enumerate}
The step (1) was proved by G.~Occhetta, L. E. Sol\'a Conde, K.~Watanabe, and J.~A. Wi{\'s}niewski \cite{OSWW}.
Following the above strategy,
we can also show that Conjecture~\ref{CP} is true for CP manifolds with Picard number one and pseudoindex four for which the evaluation morphism $e \colon U \to X$ is a smooth $\bP ^2$-fibration
(see Proposition~\ref{rhm1}).

\medskip

In the forthcoming paper, we generalize the result of this paper and prove that CP $n$-folds with Picard number $\rho \geq n-4$ are rational homogeneous manifolds.
The same result in the case where $\rho \geq n-3$ is independently obtained by K. Watanabe.

\begin{convention}
In this paper, we work over the field of complex numbers.
A morphism $f\colon X \to Y$ is called a \textit{$\bP ^r$-bundle} if it is isomorphic to the projectivization of a vector bundle of rank $r+1$.
On the other hand, a morphism $f\colon X \to Y$ is called a \textit{smooth $\bP ^r$-fibration} if it is smooth and every fiber is isomorphic to $\bP ^r$.
We will denote by
$\cO (a_1^{n_1}, \dots , a_k^{n_k})$
the vector bundle
$\cO (a_1)^{\oplus n_1} \oplus \dots \oplus \cO (a_k)^{\oplus n_k}$ 
on $\bP ^1$.
\end{convention}

\subsection*{Acknowledgement}
The author wishes to express his gratitude to Professor Yoichi Miyaoka, his supervisor, for his encouragement, comments and suggestions.
He is also grateful to Professor Hiromichi Takagi and Professor Kiwamu Watanabe for their helpful comments and suggestions.
This work was supported by the Program for Leading Graduate Schools, MEXT, Japan.

\section{Preliminaries}\label{pre}

In this section, we review some results concerning rational curves on Fano manifolds with nef tangent bundles.
Our basic references are \cite{K} and \cite{MOSWW}.

First, for brevity, we define:
\begin{definition}[{\cite[Definition 1.4]{MOSWW}}]
A manifold $X$ is said to be a \textit{CP manifold} if $X$ is a Fano manifold with nef tangent bundle.
\end{definition}

We briefly recall the known results about the Campana-Peternell conjecture in dimension five.
\begin{definition}\label{index}
Let $X$ be a Fano manifold with Picard number one.
We define the pseudoindex $i_X$ as follows:
\begin{align*}
i_X &\coloneqq \min \{\,  -K_X.C \mid \text{$C$ is a rational curve on $X$} \,\}.
\end{align*}
\end{definition}

It is known that for a CP $n$-fold ($n\geq2$) with Picard number one the pseudoindex of $X$ satisfies $3 \leq i_{X} \leq n+1$ (see e.g.\ \cite[Theorem~5.1]{W2}).
Furthermore the following holds:
\begin{enumerate}
 \item If $i_{X}=n+1$, then $X \simeq \bP ^{n}$ \cite{CMSB}.
 \item If $i_{X}=n$, then $X \simeq \bQ ^{n}$ \cite{Mi}.  
 \item If $i_{X}=3$, then $X \simeq \bP ^{2}$, $\bQ ^{3}$ or $K(G_{2})$, where $K(G_{2})$ is the $5$-dimensional contact homogeneous manifold of type $G_{2}$ \cite[Section~4]{Hw}, \cite{Mok}.
\end{enumerate}

On the other hand, K. Watanabe solved the Campana-Peternell conjecture in dimension five with Picard number greater than one \cite{W2}. Hence we have:

\begin{theorem}[{\cite[Theorem 1.2 and Corollary 5.2]{W2}}]\label{not4}
Let $X$ be a CP 5-fold.
Then one of the following holds:
\begin{enumerate}
\item $X$ is a rational homogeneous manifold.
\item $\rho _X=1$ and $i_X=4$.
\end{enumerate}
\end{theorem}

\subsection*{Families of Rational Curves on CP manifolds}
Let $X$ be a CP manifold of dimension $n$.
We will denote by
$\Hom (\bP ^1, X)$
the scheme which parametrizes morphisms from $\bP ^1$ to $X$ and by $\ev\colon  \bP ^1 \times \Hom (\bP ^1, X) \to X$ the evaluation morphism  $\ev\colon (p,[f]) \mapsto f(p) \in X$.   
As the normalized quotient of $\Hom (\bP ^1,X)$ under the action of $\Aut(\bP ^1)$,
we construct $\RC ^n(X)$, the scheme parametrizing rational curves on $X$
(see \cite[II, Section~2]{K} and \cite[2.3]{MOSWW}).

\begin{definition}\label{minirc}
An irreducible component $V$ of $\RC ^n(X)$ is called a
\textit{minimal rational component}
if $V$ parametrizes rational curves of the minimum anticanonical degree.
\end{definition}

\begin{notation}\label{min}
Given a minimal rational component $V$ of anticanonical degree $d$,
we have the diagram consisting of the two natural projections $\pi$ and $e$
\[
\begin{CD}
     U        @> e >>   X     \\
@V \pi VV                        \\
     V,
\end{CD}
\]
where $\pi \colon  U \to V$ is the universal family and $e\colon  U \to X$ is the evaluation morphism. 
\end{notation}

\begin{proposition}[{\cite[II. Theorem~1.2, Corollary~2.12, Proposition~2.14, Theorem~2.15 and Corollary~3.5.3]{K}, \cite[Proposition~2.10]{MOSWW}}]\label{fam}
In the notation above, 
we have: 
\begin{enumerate}

\item\label{fam1}
$V$ is a smooth projective variety of dimension $n+d-3$.

\item\label{fam2}
$e$ is a smooth morphism with connected fibers of relative dimension $d-2$ and $\pi$ is a smooth $\bP ^1$-fibration.

\end{enumerate}
\end{proposition}
\begin{remark}
The connectedness of the fibers of the evaluation morphism $e$ follows from the (algebraically) simply connectedness of $X$.
\end{remark}

\section{Manifolds which admit two smooth $\bP ^1$-fibrations}\label{twoP}

In this section, we study projective manifolds which admit two smooth $\bP ^1$-fibrations, based on results  \cite[Theorem~6.5]{MOS}, \cite[Theorem~5]{MOSW} and \cite{W1}, where such manifolds with Picard number two are classified.
The result of this section will be repeatedly used later.

Let $U$ be a smooth projective variety which admits 
two fibrations over smooth projective varieties $W$ and $S$:
\[
\begin{CD}
      U         @> e >>      W     \\
@V \pi VV                           \\
      S,
\end{CD}
\]
where $\pi$ is a smooth $\bP ^1$-fibration and $e$ does not contract any $\pi$-fiber.
In particular, the images of $\pi$-fibers generate a half-line $R_S \subset \cNE (W)$.
It is natural to ask:

\begin{question}
Is $R_S$ an extremal ray of $W$ ?
\end{question}
For a more general question, we refer the reader to \cite{BCD}.
Here, we give a partial answer to this question:

\begin{theorem}\label{ext}
Assume moreover that the morphism $e$ is a smooth $\bP ^1$-fibration.
Then $R_S$ is an extremal ray of $W$.
Furthermore, the contraction of $R_S$ is a smooth morphism.
\end{theorem}

\begin{remark}\label{remquot}
As a consequence, there exists the following diagram
\[
\begin{CD}
U @>e >> W\\
@V\pi VV @VVfV \\
S @>g >> Z,
\end{CD}
\]
where $f$ is the contraction of the ray $R_S$.
Then every fiber of $f \circ e = g \circ \pi$ is a Fano manifold with Picard number two which admits two smooth $\bP ^1$-fibrations.
Such varieties are complete flag manifolds by \cite[Theorem 5]{MOSW} or \cite{OSWW}. In particular, all $f$-fibers are the same and isomorphic to $\bP^1$, $\bP ^2$, $\bP ^3$, $\bQ ^3$, $\bQ ^5$ or $K(G_2)$.
Furthermore,
if $-K_e. (\pi \text{-fiber}) = -1$, then $f$-fibers are isomorphic to $\bP ^2$, $\bQ ^3$ or $K(G_2)$.
\end{remark}

\begin{proof}[Proof of Theorem~\ref{ext}]
We prove first that $R_S$ is an extremal ray, by showing that the S-rationally connected quotient $W \dashrightarrow Z$ is a \emph{morphism}.
Then we prove that the contraction of $R_S$ is smooth, following the proof of  \cite[Theorem 4.4]{SW} (see also \cite[Theorem 5.2]{DPS}).
We refer the reader to \cite[Chapter~5]{De} and \cite[Chapter IV]{K} for accounts of rationally connected quotients.
\step{Set up of the notation.}
Let $n$ be the dimension of $W$.
Given a point $x \in W$, we define:
\begin{align*}
V_0(x)&\coloneqq \{ x \}, \\
V_{m}(x)&\coloneqq e(\pi ^{-1} ( \pi (e^{-1}(V_{m-1}(x))))). 
\end{align*}
Namely, $V_m(x)$ is the set of points of $W$ that can be connected to $x$ by a $S$-chain of length $m$. Furthermore, set
$V(x)\coloneqq V_n(x)$,
that is, the set of points of $W$ that can be connected to $x$ by a $S$-chain (see \cite[Step 2 of the Proof of the Theorem]{KMM2} or \cite[First Step of the Proof of Proposition 5.7]{De}).
Set $d(x) \coloneqq  \dim V(x)$.

Similarly, for the morphisms $E$ and $\Pi$ as in the diagram below,
\[
\begin{CD}
      U \times W                     @> E\coloneqq  e \times \id >>     W\times W        \\
@V \Pi \coloneqq  \pi \times \id VV                                                                  \\
      S \times W,
\end{CD}
\]
we define:
\begin{align*}
V_{0} &\coloneqq  \varDelta \subset W \times W \text{ (the diagonal)},\\
V_{m} &\coloneqq  \bigcup _{x\in W}(V_{m}(x) \times \{ x \}) \subset W \times W\\
        &=  E  (\Pi^{-1} (\Pi(E^{-1}(V_{m-1})))),\\
V &\coloneqq V_n.
\end{align*}
Namely, $V_m$ (resp.\ $V$) is the family of $V_m(x)$ (resp.\ $V(x)$) over $W$ via the second projection $\pr_2$.
Let $d_S$ be the dimension of general $S$-equivalence classes, that is, the relative dimension of $\pr_2\colon  V \to W$.
\step{Step 1. A lemma on the sequence $V_0(x)\subset V_1(x) \subset \cdots \subset V_n(x)$.}
\begin{lemma}\label{lemV}
Let the notation be as above. Then:
\begin{enumerate}
\item\label{lemV1}
  $V_{m+1}(x)=V_{m}(x)$ if and only if the restriction of $\pi$ to $e^{-1}(V_m(x))$ is of fiber type.

\item\label{lemV2}
If $V_{m+1}(x) \neq V_{m}(x)$, then \[\dim \pi ^{-1} (\pi ( e^{-1}(V_{m}(x)) )) = \dim V_{m}(x) +2 .\]

\item\label{lemV3}
$\dim V_{m+1}(x) = \dim V_{m}(x)+1$ if and only if $V_{m+1}(x) \neq V_{m}(x)$ and the restriction of $e$ to $\pi ^{-1} ( \pi ( e^{-1}(V_{m}(x)) )) $ is of fiber type. 
Furthermore, if these equivalent conditions hold, $V_{m+2}(x)=V_{m+1}(x)$.

\item\label{lemV4}
If $V_{m+1}(x)\neq V_{m}(x)$, then $\dim V_{m}(x) = \dim V_{m-1}(x)+2$.

\item\label{lemV5}
Let $m(x)$ be the integer which satisfies $V_{m(x)+1}(x)=V_{m(x)}(x)$ and $V_{m(x)}(x)\neq V_{m(x)-1}(x)$.
Then the following hold.
\begin{enumerate}

\item If the restriction of $e$ to $\pi ^{-1} ( \pi ( e^{-1}(V_{m(x)-1}(x)) )) $ is not of fiber type,
then
\begin{align*}
\dim V (x) &= 2m(x) = \dim V_{m(x)}(x)\\
&>2m(x)-2 = \dim V_{m(x)-1}(x)\\
&>\dots \\
&>2=\dim V_1(x)\\
&>0=\dim V_0(x).
\end{align*}

\item If the restriction of $e$ to $\pi ^{-1} ( \pi ( e^{-1}(V_{m(x)-1}(x)) )) $ is of fiber type,
then
\begin{align*}
\dim V (x) &= 2m(x)-1 = \dim V_{m(x)}(x)\\
&>2m(x)-2 = \dim V_{m(x)-1}(x)\\
&>2m(x)-4 = \dim V_{m(x)-2}(x)\\
&>\dots \\
&>2=\dim V_1(x)\\
&>0=\dim V_0(x).
\end{align*}

\end{enumerate}

\end{enumerate}

\end{lemma}

\begin{proof}[Proof of Lemma~\ref{lemV}]
(\ref{lemV1})
Assume that $\pi|_{e^{-1}(V_m(x))}$ is of fiber type.
Then, since $\pi$ is of relative dimension $1$,
\[
\pi ^{-1} (\pi ( e^{-1}(V_{m}(x)) )) = e^{-1}(V_{m}(x)).
\]
Hence
\[
V_{m+1}(x) = e(\pi ^{-1} (\pi ( e^{-1}(V_{m}(x)) ))) = V_{m}(x).
\]

On the other hand, if $\pi|_{e^{-1}(V_m(x))}$ is not of fiber type, then we have
\[
\dim \pi ^{-1} (\pi ( e^{-1}(V_{m}(x)) )) = \dim V_{m}(x) +2 .
\]
Hence
\[
\dim V_{m+1}(x) = \dim e(\pi ^{-1} (\pi ( e^{-1}(V_{m}(x)) ))) \geq \dim V_{m}(x)+1.
\]

(\ref{lemV2})
This is clear from (\ref{lemV1}).

(\ref{lemV3})
If $\dim V_{m+1}(x) = \dim V_{m}(x)+1$,
then by (\ref{lemV2}) we have
\[
\dim \pi ^{-1} (\pi ( e^{-1}(V_{m}(x)) )) = \dim V_{m}(x) +2.
\]
Hence
\[
\dim \pi ^{-1} (\pi ( e^{-1}(V_{m}(x)) )) > \dim V_{m+1} (x),
\]
that is,
$e|_{\pi ^{-1} ( \pi ( e^{-1}(V_{m}(x)) ))} $ is of fiber type.

On the other hand, assume that $V_{m+1}(x) \neq V_{m}(x)$ and the restriction of $e|\pi ^{-1} ( \pi ( e^{-1}(V_{m}(x)) )) $ is of fiber type.
Then, by (\ref{lemV2}),
\[
\dim \pi ^{-1} (\pi ( e^{-1}(V_{m}(x)) )) = \dim V_{m}(x) +2 .
\]
Hence $ \dim V_{m+1}(x) = \dim V_{m}(x)+1 $.
In this case, we have 
\[
e^{-1}(V_{m+1}(x))=\pi ^{-1} (\pi ( e^{-1}(V_{m}(x)) )), 
\]
and hence
$V_{m+2}(x)=V_{m+1}(x) $.

(\ref{lemV4})
Since relative dimensions of $e$ and $\pi$ are $1$,
\[\dim V_{m}(x) \leq \dim V_{m-1}(x)+2.\]

If $\dim V_{m}(x) = \dim V_{m-1}(x)$,
then $V_{m}(x) = V_{m-1}(x)$, and hence we have $V_{m+1}(x) = V_{m}(x)$.

If $\dim V_{m}(x) = \dim V_{m-1}(x) + 1$,
then $e|_{\pi ^{-1} ( \pi ( e^{-1}(V_{m-1}(x)) ))} $ is of fiber type by (\ref{lemV3}).
Hence
$ e^{-1}(V_{m}(x)) = \pi ^{-1} ( \pi ( e^{-1}(V_{m-1}(x)) ))$.
Therefore $V_{m+1}(x) = V_{m}(x)$

(\ref{lemV5})
This is clear from (\ref{lemV3}) and (\ref{lemV4}).

\end{proof}

Fix a general point $x \in W$ and $m_S\coloneqq m(x)$, then $\dim V(x) = d_S$ and the restriction of $\Pi$ to $E^{-1}(V_{m_S})$ is of fiber type. 

\step{Step 2.}
In this step, we prove that $m(y) = m_S$ for every $y \in W$.
Since the restriction of $\Pi$ to $E^{-1}(V_{m_S})$ is of fiber type,
the restriction of $\pi$ to $e^{-1}(V_{m_S}(y))$ is of fiber type for every $y \in W$.
Hence $m(y)\leq m_S$ by Lemma~\ref{lemV}~(\ref{lemV1}).
On the other hand, by semicontinuity,
$\dim V(y) \geq d_S$ for every $y \in Y$.
It follows that $m(y) \geq m_S$ by Lemma~\ref{lemV}~(\ref{lemV5}).
Hence $m(y) = m_S$.

\step{Step 3.}
We prove that $d(y) = d_S$ for every $y \in W$.
First, assume that $d_S$ is even.
Then $\dim V(y) \geq d_S = 2m_S$.
By Step 2, $m(y)=m_S$. Hence $d(y) = d_S$ by Lemma~\ref{lemV}~(\ref{lemV5}).

Next, 
assume that $d_S$ is odd.
By Lemma~\ref{lemV}~(\ref{lemV3}), the restriction of $E$ to $\Pi ^{-1} ( \Pi ( E^{-1}(V_{m_S-1}) )) $ is of fiber type.
Hence the restriction of $e$ to $\pi ^{-1} ( \pi ( e^{-1}(V_{m_S-1})(y) )) $ is of fiber type for every $y \in W$.
This implies that $\dim V_{m_S}(y)=\dim V_{m_S-1}(y)+1$ by Lemma~\ref{lemV}~(\ref{lemV3}).
Hence $d(y) = d_S$ by Lemma~\ref{lemV}~(\ref{lemV5}).

\step{Step 4. The contraction of $R_S$.}
By Step 3, we have a morphism $f \colon W \to \Chow (W)$ such that $\pr_2 \colon  V \to W$ is the pullback of the universal family by $f$ \cite[I. Theorem~3.17 and Theorem~3.21]{K}.
We will denote by $Z$ the normalization of $\im f$. Then we have the $S$-rationally connected quotient morphism $f \colon W \to Z$. Hence $R_S$ is an extremal ray and $f$ is the contraction of $R_S$ (see also \cite[Proposition~1]{BCD}).

\step{Step 5. Smoothness of the contraction.}
By rigidity lemma, we have the following diagram
\[
\begin{CD}
U @>e >> W\\
@V\pi VV @VVfV \\
S @>g >> Z.
\end{CD}
\]
By symmetry, $g$ is the $W$-rationally connected quotient morphism of $S$.

Now, $f$ and $g$ are equidimensional Mori contractions with irreducible fibers. Furthermore, $e^*T_W$ (resp.\ $\pi ^* T_W$) is $\pi$-nef (resp.\ $e$-nef).
A similar argument as in the proof of \cite[Lemma~4.12]{SW} shows that every $f$-fiber with its reduced structure is a Fano manifold with trivial normal bundle (see below). Hence the contraction of $R_S$ is smooth by \cite[Lemma~4.13]{SW}.

\medskip
We give an outline of a proof of the fact that every $f$-fiber with its reduced structure is a Fano manifold with trivial normal bundle.

General fibers of $h\coloneqq  f\circ e$ are complete flag manifolds whose Picard numbers are two by \cite[Theorem 5]{MOSW} or \cite{OSWW}.
We may assume that the relative dimension of $f$ is one or that $f$ is corresponding to the long root i.e.\ $-K_W. (\pi \text{-fiber}) = 3$.
Fix an arbitrary point $z \in Z$ and let $W_z$ (resp.\ $S_z$) be the $f$-fiber (resp.\ $g$-fiber) over $z$.
By the same argument as in the proof of \cite[Lemma~4.12]{SW},
we see that $W\zred $ and $S\zred $ are smooth and the normal bundle $N_{W\zred  /W}$ (resp.\ $N_{S\zred /S}$) is trivial over any $\pi$-fiber (resp.\ $e$-fiber).
Hence we may assume that the relative dimension of $f$ is greater than one.
In this case $W\zred$ is corresponding to the long root and $S\zred$ is the unsplit family of lines on $W\zred$.
Indeed, $c_1(N_{W\zred /W})$ is trivial and, by adjunction, $-K_W|_{W\zred }=-K_{W\zred }$.
Now, $N_{W\zred /W}$ is trivial by \cite[Proposition 1.2]{AW}.
\end{proof}

\section{Families of minimal rational curves}\label{fammrc}

Let $X$ be a CP manifold with Picard number one and pseudoindex four,
and let $(U,X,V; e,\pi)$ be as in Notation~\ref{min}.
In this section, we study the structure of $(U,X,V; e,\pi)$.
First, we prove a proposition and a lemma (Proposition~\ref{rho} and Lemma~\ref{-1}).

\begin{proposition}[{\cite[Remark 3.7]{C}}]\label{rho}
Let $f\colon Z \to S$ be a smooth Mori contraction over a smooth rationally connected variety $S$.
Then $\rho(F) = \rho(Z) - \rho(S)$ for every fiber $F$.
\end{proposition}

\begin{proof}
Since $F$ is a Fano manifold, $F$ is rationally connected by \cite{KMM3}.
Furthermore, by \cite{GHS}, $Z$ is rationally connected.
Hence
\begin{align*}
&\Pic (Z) \otimes \bQ  \simeq H^2(Z, \bQ ),\\
&\Pic (F) \otimes \bQ  \simeq H^2(F, \bQ ).
\end{align*}
Since $S$ is simply connected, the monodromy action is trivial.
Hence
\[
H^2(Z, \bQ ) \to H^2(F, \bQ )
\]
is surjective by Deligne's invariant cycle theorem.
Therefore,
\[
\rho(F) =\dim N_1(F,Z).
\]
On the other hand, $\dim N_1(F,Z) = \rho(Z) - \rho(S)$ by \cite[Lemma 3.3]{C}.
\end{proof}

\begin{lemma}\label{-1}
Let $X$ be a CP manifold with Picard number one and pseudoindex four, $e\colon  U \to X$ the evaluation morphism of minimal rational curves and $F$ an arbitrary $e$-fiber.
Then the following hold:

\begin{enumerate}

\item\label{-11} $H_2(F, \bQ ) \to H_2(U, \bQ )$
is injective.

\item\label{-12} $H^{1,1}(U, \bQ ) \to H^{1,1}(F, \bQ )$ is surjective.

\item\label{-13} For distinct $(-1)$-curves $C_i \subset F$ $(i=1,2)$,
$[C_1] \not \nequiv [C_2]$ in $N_1(U)$.

\end{enumerate}

\end{lemma}

\begin{proof}
The first and second assertions are consequences of Deligne's invariant cycle theorem.

(\ref{-13}) Since $[C_1] \not \nequiv [C_2]$ in $N_1(F)$,
there exists a line bundle $\sL$ on $F$ such that $\deg_{C_1}\sL  \neq \deg_{C_2}\sL $.
Denote by $[\sL]$ its image in $H^{1,1}(F, \bQ )$.
By (\ref{-12}), there exists $D \in H^{1,1}(U, \bQ )$ whose restriction to $F$ is $[\sL]$.
Furthermore, by Lefschetz (1,1)-Theorem, there exists $\overline{\sL} \in \Pic (U) \otimes \bQ$ whose image in $H^{1,1}(F, \bQ )$ is $D$.
By compatibility, $\overline{\sL} |_F\nequiv \sL$.
Hence the assertion follows.

\end{proof}

The main result of this section is the following:
\begin{theorem}\label{fib}
Let X be a CP manifold with Picard number one and pseudoindex four.
Then one of the following holds:
\begin{enumerate}

\item\label{fib1} The evaluation morphism $e$ is a smooth $\bP ^2$-fibration.

\item\label{fib2} There exists the commutative diagram
\[
\begin{CD}
     U             @>f >>         W              @> g >> X\\
@V\pi VV                     @V q VV \\
     V             @>p >>        Y
\end{CD}
\]
with the following properties:

\begin{enumerate}

\item $e=g \circ f$,
\item $f$ and $g$ are smooth $\bP ^1$-fibrations,
\item $p$ and $q$ are smooth elementary Mori contractions.
\end{enumerate}

\end{enumerate}
\end{theorem}

For the proof, we use the following result due to K.~Oguiso and E.~Viehweg.
\begin{theorem}[{\cite[Theorem 0.1]{OV}}]\label{isotriv}
All smooth projective families of minimal surfaces of non-negative Kodaira dimension over elliptic curves or  over $\bC ^*$ are isotrivial.
\end{theorem}

\begin{lemma}\label{ray}
For a CP manifold $X$ with Picard number one and pseudoindex four,
there exists a $K_U$-negative curve contained in an $e$-fiber.
In particular, there exists a $K_U$-negative extremal ray $R$ of $\cNE (U)$ which is contracted by $e$.
\end{lemma}

\begin{proof}
The idea of this proof is in \cite[the proof of Lemma~1.2.2]{Mok}.

If $K_U$ is $e$-nef, then every $e$-fiber is a minimal surface of  non-negative Kodaira dimension.
Hence, for every rational curve $C \subset X$, $e$ is isotrivial over $C$ by Theorem~\ref{isotriv}.
Let $C$ be a rational curve parametrized by $V$,
then $U \times _X C$ is an isotrivial family of minimal surfaces.
Furthermore, there exists a section $s:C \to  U \times _X C$ corresponding to $[C] \in V$, which is contracted by $\pi \colon U \times _X C \to V $.
Hence the image $\pi (U \times _X C)$ is two dimensional and $\pi (e^{-1}(x))$ is independent of $x \in C$.
Moreover, since $\rho (X) =1$, $X$ is $V$-rationally connected by \cite[Lemma~3]{KMM2}.
Hence $\dim V =2$.
This contradicts Proposition~\ref{fam}~(\ref{fam1}).
Hence $K_U$ is not $e$-nef.
The last assertion follows from the cone theorem.

\end{proof}

\begin{proof}[Proof of Theorem~\ref{fib}]
By Lemma~\ref{ray}, we have the contraction $\phi$ of $R$ by the contraction theorem \cite[Theorem 3-2-1]{KMM}
\[
\begin{CD}
U @>\phi>> W' @>\psi>> X,
\end{CD}
\]
where $e= \psi \circ \phi $.

\case{Case 1. $W' \simeq X$.}
In this case, $e$ is an elementary Mori contraction and every fiber of $e$ is isomorphic to $\bP ^2$ by Proposition~\ref{rho}.
Hence the case (\ref{fib1}) of Theorem~\ref{fib} occurs.

\case{Case 2. $W' \not \simeq X$.}
In this case, every fiber of $\phi$ has dimension at most $1$. Such morphisms are classified in \cite{A} (see also \cite[Theorem 4.2.1]{AM} for the statement).

\case{Subcase 2.1. $\phi$ is a birational morphism.}
In this case, by \cite{A}, $W'$ is a smooth projective variety and $\phi$ is a blow-up of a smooth codimension two subvariety $Z \subset W' $. 
We will denote by $E$ the exceptional divisor.
We prove that every $e$-fiber is isomorphic to $\bF _1$ and the case (\ref{fib2}) of Theorem~\ref{fib} occurs.

\step{Step 1.}
$\psi _Z\colon Z \to X$ is finite and hence surjective.
Otherwise, there would exist a curve $D \subset Z$ contracted by $\psi$. Then
$e^{-1}(\psi(D)) = \phi ^{-1}(D)$ by the dimensional reason. Hence $\psi^{-1}(\psi(D))=D$, contradicting the fact that $\psi$ is of relative dimension two.  

\step{Step 2.}
$\psi_Z\colon Z \to X$ is an isomorphism.
Indeed, $\psi_Z\colon Z \to X$ is generically one-to-one by Lemma~\ref{-1}~(\ref{-13}).

\step{Step 3.}
We prove that $\pi _E\colon E \to V$ is surjective.
Otherwise, $\pi _E\colon E \to V$ is a contraction of fiber type.
Since $\pi$ is a smooth $\bP ^1$-fibration,
$\pi _E\colon E \to \pi (E)$ is a smooth $\bP ^1$-fibration and $\pi (E)$ is smooth.
In \cite{W1}, such varieties are completely classified, and we have $Z \simeq \bP ^3$ since $i_X = i_Z = 4$.
However, the universal family of lines on $\bP ^3$ is a $\bP ^2$-bundle.
This contradicts the fact that $\phi$ is a birational morphism.

\step{Step 4.}
We have $N_1(E) \simeq N_1(V)$ and $\cNE (E) \simeq \cNE (V)$ by $(\pi _ E )_*$.
Indeed, $\rho_V = 2$ since $\pi _E\colon E \to V$ is surjective and $\rho_V \geq 2$.

\step{Step 5. Contraction of $V$.}
By Step 4, the ray corresponding to the contraction $\phi _E$ defines a ray $R'$ of $\cNE (V)$ and we have the following diagram
\[
\begin{CD}
     E        @>\phi _E>>     Z                                  \\
@VVV                          @VVV                                 \\
     U        @>\phi>>         W'             @>\psi>> X  \\
@V\pi VV                    @V \beta VV                       \\
     V        @>\alpha>>     M,
\end{CD}
\]
where $\alpha$ is the contraction of the ray $R'$.

\step{Step 6.}
We have $\dim M = n$.
Indeed, $\beta _Z\colon  Z \to M$ is a finite surjective morphism since it is surjective and $\rho _Z = \rho _M =1$.

\step{Step 7.}
In this step, we prove that every fiber of $e$ is isomorphic to $\bF _1$ and $\psi$ is a smooth $\bP ^2$-fibration by a similar argument as in the proof of Lemma~\ref{ray}.
If $K_W'$ is $\psi$-nef, then $\psi$ is isotrivial on every rational curve on $X$ by Theorem~\ref{isotriv}. Hence $\dim M = 2$.
This contradicts Step 6. 
Therefore $\psi$ is an elementary Mori contraction, and hence every fiber of $\psi$ is isomorphic to $\bP ^2$ by Proposition~\ref{rho}.

\step{Step 8. Conclusions.}
By Step 7, $e$ is a Mori contraction.
Hence we get another factorization $e\colon U \xrightarrow{f} W \xrightarrow{g} X$.
The same argument as in Subcase 2.2 below shows that $f$ and $g$ are smooth $\bP ^1$-fibrations.

\case{Subcase 2.2. $f$ is a fiber type contraction.}
In this case, by \cite{A}, $f$ is a flat conic bundle, $W'$ is a smooth variety and each fiber $F$ of $f$ is one of the following:
\begin{enumerate}

\item\label{conic1} $F \simeq \bP ^1$ and $-K_U.F = 2$.

\item\label{conic2} $F \simeq C_1 \cup C_2$ and $-K_U.C_i = 1$, where $C_1$ and $C_2$ are smooth rational curves.

\item\label{conic3} $F\red \simeq \bP ^1$,
$-K_U.F\red =1$,
$F$ is a nonreduced conic
and $N_{F\red /U} \simeq \cO (1, -1^2, 0^{n-2} )$ or $\cO (1, -2, 0^{n-1} )$.

\end{enumerate}

Set $f\coloneqq \phi$, $g\coloneqq \psi$ and $W\coloneqq W'$.
We prove that the case (\ref{fib2}) of Theorem~\ref{fib} occurs.

\step{Step 1.}
If there is a fiber of type (\ref{conic2}), then $C_1 \not \nequiv C_2$ in $N_1(U)$ by Lemma~\ref{-1}~(\ref{-13}).
This contradicts the fact that $f$ is an elementary contraction.
Furthermore, for a $(-1)$-curve $C$ in a fiber, $N_{C/U} \simeq \cO(-1, 0^{n} )$.
Hence there is no fiber of type (\ref{conic3}).
Therefore, $f$ is a smooth $\bP ^1$-fibration.

\step{Step 2.}
In this step, we prove that $g$ is a smooth $\bP ^1$-fibration by a similar argument as in the proof of Lemma~\ref{ray}.
By Step 1, $g$ is a smooth fibration of relative dimension one.
On the other hand, by Step 1 and Theorem~\ref{ext}, there exists the $V$ (resp.\ $W$)-rationally connected quotient morphism $q\colon W\to Y$ (resp.\ $p\colon V \to Y$).

Assume that the genus of $g$-fibers are positive,
then $g$ is isotrivial on any rational curve on $X$.
Hence $\dim Y =1$ and the relative dimension of $q$ is $n$.
Therefore, the restriction $g|_{q\text{-fiber}}$ is finite and surjective onto $X$.
Note that any $q$-fiber is rational homogeneous manifold as in Remark~\ref{remquot} and that $i_X=4$.
Hence, by \cite{HM} or \cite{L}, we have $X \simeq \bP ^3$.
This gives a contradiction.
Hence $g$ is a smooth $\bP^1$-fibration, completing the proof.

\end{proof}

\section{Proof of Theorem~\ref{CP5}}
\label{pf_CP5}

First, we prove the proposition below, following 
the strategy of \cite{MOSW}.

\begin{proposition}\label{rhm1}
Let $X$ be a CP manifold with Picard number one and pseudoindex four.
Assume that the evaluation morphism $e$ is a smooth $\bP^{2}$-fibration.
Then $X$ is a rational homogeneous manifold.
In particular, $X$ is isomorphic to $\bP ^3$ or the Lagrangian Grassmannian $LG(3,6)$.
\end{proposition}

\begin{remark}
In this proposition, we \emph{do not} assume that $X$ 
has dimension five.
Note that $\dim LG(3,6) = 6$.

\end{remark}

\begin{proof}[Proof of Proposition~\ref{rhm1}]
Fix an arbitrary $\pi$-fiber $C$.

\begin{claim}
$T_e|_C \simeq \cO (-1) \oplus \cO (-1)$.
\end{claim}

\begin{proof}[Proof of Claim]
Consider the base change of $e$ over $C$:
\[
\begin{CD}
    U_C       @>e_C >>         C     \\
@V VV                            @VVV   \\
    U           @>e >>             X     \\
@V\pi VV                                    \\
    V.
\end{CD}
\]
Then there exists a natural section $s\colon C \to U_C$ corresponding to $C \subset U$.
Since $e$ is a $\bP ^2$-fibration,
there exists a rank $3$ vector bundle
$ \sF \simeq \cO (a,b,c)$
on $C$ such that $\bP (\sF) \simeq U_C$.
Furthermore, there exists a quotient
$ \sF \to \sL \simeq \cO (\ell )$ corresponding to the section $s$.
We may assume that $a\geq b \geq c$ and $\sF$ is normalized
i.e.\ $c_1(\sF )=0$, $1$ or $2$.
We will denote by $\xi$ the tautological divisor on $\bP (\sF)$.

Since $\pi|_{U_C}$ contracts section $s(C)$,
we have $\sL  \simeq  \cO (c) $.
Furthermore, $a\geq b > c$. Indeed, $\pi|_{U_C}$ contracts only curves.

Since $K_{U}.C = -2$ and $K_{X}.e(C) = -4$, we have $K_{e}.C = 2$.
It follows that $0= 3c+2-c_1(\sF )$ since $-K_{e} \simeq 3 \xi + e_C^{*}(-c_1(\sF ))$ on $U_C$.
Therefore, $c_1(\sF )=2$ and $c=0$.
This implies that $\sF \simeq \cO (1^2,0)$.

\end{proof}

Set $M \coloneqq  \bP (T_e)$,
and we denote by $p$ the projection $M \to U$.
Then every fiber of $e \circ p$ is isomorphic to $\bP (T_{\bP ^2})$.
Note that $\bP (T_{\bP ^2})$ is a CP manifold.
Hence, by \cite[Theorem 4.4]{SW}, there exists a smooth $\bP ^1$-fibration $q\colon M \to N $ which is different from $p\colon M \to U$, and we have the following diagram:
\[
\begin{CD}
     M          @>q >>       N    \\
@Vp VV                     @VVV  \\
     U          @>e >>       X.
\end{CD}
\]

Also, since $T_e|_C \simeq \cO (-1) \oplus \cO (-1)$, every fiber of $\pi \circ p$ is isomorphic to $\bP ^1 \times \bP ^1$ which is a CP manifold.
Hence, there exists a third smooth $\bP ^1$-fibration $r\colon M \to L $ which is different from $p$ and $q$:
\[
\begin{CD}
     L     @<r<<              M          @>q >>       N        \\
@V VV                     @V p VV                     @VVV  \\
     V     @<\pi <<         U          @>e >>        X.
\end{CD}
\]
Furthermore, by Theorem~\ref{ext} for $q$ and $r$, there exists the quotient
\[
\begin{CD}
    M         @>q >>       N     \\
@Vr VV                   @VVV   \\
     L        @> >>          Q
\end{CD}
\]
and $M \to Q$ is a contraction of a $2$-dimensional face of $\cNE (M)$.
This implies that $M$ is a Fano manifold whose elementary contractions are smooth $\bP ^1$-fibrations (because $\rho _M =3$).
Hence $M$ is a rational homogeneous manifold by \cite{OSWW} , and so is $X$. 
\end{proof}

\begin{lemma}\label{neg}
If the evaluation morphism $e$ is as in Theorem~\ref{fib}~(\ref{fib2}),
then 
\[
-K_f.C = -K_g.f(C) = -1
\]
for any $\pi$-fiber $C$.
\end{lemma}

\begin{proof}
Since $-K_e.C=-2$,
it is enough to see that 
(i) $-K_f.C<0$
and
(ii) $-K_g.f(C)<0$.
We apply a similar argument as in the proof of Lemma~\ref{ray}.

(i) The surface $U \times _W C$
is isomorphic to $\bF_m$ and there exists a section $C'$ over $C$ corresponding to $C \subset U$.
Since $V$ is a family of rational curves, $\pi \colon U \times _W C \to \pi (U \times _W C)$ is generically finite.
Hence we have $m\neq 0$ and $C'$ is the negative section of this Hirzebruch surface.

(ii) 
Consider the Hirzebruch surface
$W \times _X C$.
First, we prove that
$q\colon  W \times _X C \to q(W \times _X C)$ is generically finite for some $C$ (cf.\ Step 2 of Subcase 2.2 in the proof of Theorem~\ref{fib}).
Otherwise, $\dim Y = 1 $ and the relative dimension of $q$ is $n$.  
Therefore, the restriction $g|_{q\text{-fiber}}$ is finite and surjective onto $X$.
Any $q$-fiber is a rational homogeneous manifold as in Remark~\ref{remquot} and $i_X=4$.
Hence, by \cite{HM} or \cite{L}, we have $X \simeq \bP ^3$.
However, the evaluation morphism of lines on $\bP ^3$ is a smooth $\bP ^2$-fibration.
This contradicts our assumption.

Hence $q\colon  W \times _X C \to Y$ is generically finite for some $C$.
Furthermore $\pi$-fiber $C$ defines a section of the projection $W \times _X C \to C$ which is contracted by $q$. Hence the assertion follows.

\end{proof}

\begin{proof}[Proof of Theorem~\ref{CP5}]

By Theorem~\ref{not4}, we may assume that $X$ has Picard number one and pseudoindex four.
Furthermore, by Proposition~\ref{rhm1}, there is no CP $5$-fold as in Theorem~\ref{fib}~(\ref{fib1}).
Hence, we may assume that the evaluation morphism of minimal rational curves is as in Theorem~\ref{fib}~(\ref{fib2}).

In this case, by Lemma~\ref{neg}, we have $-K_f.(\pi \text{-fiber}) = -1$.
Hence the morphism $q$ is a smooth morphism of relative dimension $2$, $3$ or $5$,  and fibers are $\bP ^2$, $\bQ ^3$ or $K(G_2)$ by Remark~\ref{remquot}.

First, assume that $q$ is a smooth morphism of relative dimension $2$. Then $q$ is a smooth $\bP ^2$-fibration.
Since $W$ admits two smooth $\bP ^r$-fibrations $g$ and $q$, we have:
\begin{align*}
b_4(W)&=1+b_2(Y)+b_4(Y)=b_2(X)+b_4(X),\\
b_6(W)&=b_2(Y)+b_4(Y)+b_6(Y)=b_4(X)+b_6(X).
\end{align*}
Furthermore, by duality, 
$b_4(X)=b_6(X)$ and $b_2(Y)=b_6(Y)$.
Hence these equations give $b_4(X)=1$ and $b_4(Y)=0$.
This contradicts the fact that $b_4(Y) \geq 1$.

Hence, we may assume that $q$ is a smooth morphism of relative dimension $3$ or $5$.
Then, $-K_{g}$ is nef by Proposition~\ref{nef} below.
This contradicts Lemma~\ref{neg}. 
\end{proof}

\begin{proposition}\label{nef}
Let X be a smooth Fano manifold of dimension $2n-1$ whose Picard number is one and $g\colon W\to X$ a smooth $\bP ^1$-fibration over $X$.
Assume that there exists another nontrivial contraction $g\colon W\to Y$ onto a variety $Y$ of dimension $\leq n$.
Then $-K_g$ is nef.
\end{proposition}

\begin{proof}
Since $g$ is a smooth $\bP ^1$-fibration over $X$, there exists a vector bundle $\sG $ of rank $3$ over $X$ which have the following properties:

\begin{enumerate}
\item $W \subset \bP (\sG )$ and $W \in |2\eta|$, where $\eta$ is the tautological divisor on $\bP (\sG )$.

\item $\sG  \simeq \sG  ^*$.

\end{enumerate}
Then, by adjunction, $-K_g = \eta_W$.
Moreover $\eta^3 \nequiv -g^*(c_2(\sG )).\eta$ by the definition of Chern classes.

Let $H$ be the pullback of the ample generator of $\Pic (X)$ by $g$, and
let $\tau$ be the slope of $g$, that is, the real number $\tau$ which satisfies $-K_g+\tau H$ is nef but not ample (cf.\ \cite{MOS}).
By definition, $\eta_W+\tau H$ $(=-K_g+\tau H)$ is a pullback of a $\bR$-divisor on $Y$.
Hence
\[
(\eta_W+\tau H)^{2n}=\cdots=(\eta_W+\tau H)^{n+1}=0.
\]
Therefore, on $\bP (\sG )$,
\[
(\eta+\tau H)^{2n}.\eta =\cdots=(\eta+\tau H)^{n+1}.\eta.H^{n-1}=0.
\]
Using the relations $\eta^3 =  -g^*(c_2(\sG )).\eta$ and $H^{2n}=0$, we have
\[
\sum^{n}_{i=1}
\binom{2n-j+1}{2i-1}M_{2i-1} \tau^{2n-2i-j+2}
=0
\]
for all $j=1,2,\dots n$, where $M_{2i-1}=  (-g^*(c_2(\sG )))^{i-1}.H^{2n-i}$.

Since $M_1\neq 0$,
\[
\det
\begin{pmatrix}
\binom{2n-j+1}{2i-1}\tau^{2n-2i-j+2}
\end{pmatrix}
_{i,j}
=0.
\]
On the other hand,
\[
\det
\begin{pmatrix}
\binom{2n-j+1}{2i-1}
\end{pmatrix}
_{i,j}
\neq
0.
\]
Hence $\tau =0$, completing the proof. 
\end{proof}


\bibliographystyle{amsplain}
\bibliography{}

\end{document}